\documentclass[11pt]{amsart}
\usepackage[english]{babel}
\usepackage{amsfonts,latexsym,amsthm,amssymb,graphicx}
\usepackage{mathabx}
\usepackage[all]{xy}
\usepackage[usenames]{color}
\usepackage{tikz}
\usetikzlibrary{shapes}
\usetikzlibrary{decorations.pathreplacing}
\usepackage{amsmath}
\usepackage{graphicx}
\usepackage[enableskew]{youngtab}
\usepackage[utf8]{inputenc}
\usepackage[english]{babel}
\usepackage{tikz}
\usetikzlibrary{arrows}
\usepackage{float}
\usepackage{amscd}
\usepackage{pstricks}
\usepackage{tableau}
\usepackage{multicol}
\usepackage{geometry}
\usepackage[alphabetic]{amsrefs}
\usepackage{enumerate}
\usepackage{pdfsync}
\usepackage[colorlinks=true, linkcolor=blue, citecolor=blue, urlcolor=blue]{hyperref}
\usepackage{dynkin-diagrams}
\usetikzlibrary{decorations.pathreplacing}

\usepackage{xcolor}

\usepackage{ytableau}

\DeclareMathOperator{\Fl}{Fl}

\newcommand{\bx}{{\color{blue}x}}

\newcommand{\Wlarge}{W}

\newcommand{\QH}{\mathrm{QH}}

\DeclareFontFamily{OT1}{rsfs}{}
\DeclareFontShape{OT1}{rsfs}{n}{it}{<-> rsfs10}{}
\DeclareMathAlphabet{\mathscr}{OT1}{rsfs}{n}{it}

\CompileMatrices

\newtheorem{thm}{Theorem}[section]
\newtheorem{lemma}[thm]{Lemma}

\newtheorem{prop}[thm]{Proposition}

{\theoremstyle{definition} \newtheorem{defn}[thm]{Definition}}
{\theoremstyle{remark} \newtheorem{remark}[thm]{Remark}
\newtheorem{example}[thm]{Example}}

\begin{document}

\title[]{Minimum quantum degrees with Maya diagrams}



\author{Ryan M. Shifler}

\address{
Department of Mathematical Sciences,
Henson Science Hall, 
Salisbury University,
Salisbury, MD 21801
}
\email{rmshifler@salisbury.edu}


\begin{abstract}
We use Maya diagrams to refine the criterion by Fulton and Woodward for the smallest powers of the quantum parameter $q$ that occur in a product of Schubert classes in the (small) quantum cohomology of partial flags. Our approach using Maya diagrams yields a combinatorial proof that the minimal quantum degrees are unique for partial flags. Furthermore, visual combinatorial rules are given to perform precise calculations.
\end{abstract}

\maketitle


%
%

\section{Introduction}
Let $I=\{i_0:=0 < i_1<i_2< \cdots<i_k < i_{k+1}:=n\}$. Let $\Fl:=\Fl(I;n)$ denote the partial flag given by \[ \Fl(I;n):=\{0 \subset V_1 \subset V_2 \subset \cdots \subset V_k \subset \mathbb{C}^n: \dim V_j=i_j \}.\] Let $\QH^*(\Fl)$ be the small quantum cohomology with Schubert classes $\sigma_w$, $w \in W^{P}$. The set $W^{P}$ is the minimum length coset representative of the associated Weyl group $W$ modded out by a parabolic $P$ that corresponds to the set $I$. The set $W^P$ is defined in Section \ref{preliminaries}.  We denote the Poincare dual of $\sigma_v$ by $\sigma^v$ or $\sigma_{v^\vee}$. The small quantum cohomology ring $\QH^*(\Fl)$ is a graded $\mathbb{Z}[q]$-module. Multiplication is given by
\[ \sigma^v \star \sigma_w= \sum_{u,d\geq 0} c_{v^\vee,w}^{u,d}q^d\sigma_u\] where $c_{v^\vee,w}^{u,d}$ is the Gromov-Witten invariant that enumerates the rational curves of degree $d$. Given any element $\tau \in \QH^*(\Fl)$, we say that $q^d$ {\bf occurs} in $\tau$ if the coefficient of $q^d\sigma_w$ is not zero for some $w \in W^{P}$.

The purpose of the article is to use Maya diagrams to refine a criterion by Fulton and Woodward in \cite{FW} for the smallest powers of the quantum parameter $q$ that occur in a product of Schubert classes in the (small) quantum cohomology of partial flag. Using the moment graph, this requires many cases to be checked to find the minimum quantum degree. Maya diagrams reduce this process to a single calculation. The Maya diagram approach can be thought of as a generalization of removing rim hooks on the Young Tableau in the Grassmannian case presented in \cite{BERTRAM1999728}. The results in this article are combinatorial in the sense that we use Maya diagrams to describe the chains in the moment graph that Fulton and Woodward defined in \cite{FW} (see Proposition \ref{prop:chains}).

Furthermore, our approach using Maya diagrams yields a combinatorial proof that the minimal quantum degrees are, in fact, unique for partial flags. With geometric techniques, Postnikov proved that the minimum quantum degree is unique for $G/B$ in \cite{postnikov:qbruhat}*{Corollary 3}. This result was later extended to the general homogeneous space $G/P$ in \cite{BCLM:EulerChar} also using geometric techniques.  Minimum quantum degrees are also studied in \cites{postnikov:affine,Buch,yong, Belkale,Bar,ShiflerWithrow}.

Maya diagrams for partial flags give a characterization of the Bruhat order by slightly modifying a theorem by Proctor in \cite{ProctBO}*{Theorem 5A} and are stated herein as Proposition \ref{prop:bruhatorder}. Furthermore, the notion of generalized rim hooks is defined in Definition \ref{defn:genrimhook}. With these two key notions, there is a canonical lower bound for the minimal quantum degrees as stated in the Lemma \ref{lem:lbmdeg}. We then show that this lower bound is achieved in Theorem \ref{thm:lowerboundach}. Next, we begin with preliminaries to state and prove our main results. We expect analogous results to hold for types B, C, and D. Since the Weyl groups are different for each type, the definition of the Maya diagrams will need to be modified. Next, we begin with preliminaries to state and prove our main results.

{\em Acknowledgements.} I would like to thank Hiroshi Naruse for very useful correspondences. I would also like to thank the anonymous referees for very useful suggestions. 

\section{Preliminaries} \label{preliminaries}

Let $I=\{i_0:=0 < i_1<i_2< \cdots<i_k < i_{k+1}:=n\}$. Let $\Fl:=\Fl(I;n)$ denote the partial flag given by \[ \Fl(I;n):=\{0 \subset V_1 \subset V_2 \subset \cdots \subset V_k \subset \mathbb{C}^n: \dim V_j=i_j \}.\]

Consider the root system of type $A_{n-1}$ with positive roots $R^+=\{e_l-e_m: 1 \leq l < m \leq n \}$ and the subset of simple roots $\Delta=\{\alpha_l:=e_l-e_{l+1}:1 \leq l \leq n-1 \}$. The associated Weyl group $W$ is $S_n.$ For $1 \leq l \leq n-1$ denote by $s_l$ the simple reflection corresponding to the root $e_l-e_{l+1}$. Each $I=\{i_0:=0 < i_1<i_2< \cdots<i_k <i_{k+1}:=n\}$ determines a parabolic subgroup $P_I$ with the Weyl group $W_{P_I}=\left<s_l: l \neq i_j \right>$ generated by reflections with indices not in $I$. We will define $P:=P_I$ for notational ease. Let $\Delta_P:=\{ \alpha_{i_s}: i_s \notin \{i_1, \cdots, i_k \} \}$ and $R^+_P:=\mbox{Span} \Delta_P \cap R^+$; these are the positive roots of $P$. Let $\alpha \in R^+ \backslash R^+_P$. Then $\alpha+\Delta_P$ is the sum of simple roots in $R^+ \backslash R^+_P$ given by \[\alpha + \Delta_P= \sum_{j=1}^k d_j(e_{i_{j}}-e_{i_{j+1}}) +\Delta_P \] For notational ease, we will denote this sum by the $k$-tuple $(d_1,d_2,\cdots,d_k)$. Let $\ell: W \rightarrow \mathbb{N}$ be the length function and denote by $W^{P}$ the set of minimal length representatives of the cosets of $W/W_{P}$. The length function descends to $W/W_{P}$ by $\ell(uW_{P})=\ell(u')$ where $u' \in w^P$ is the minimal length representative of the coset $uW_{P}$. We have a natural ordering $1<2< \cdots <n$. Since $w(i_k+1)<\cdots<w(n)$ are determined, we will identify the elements of $W^{P}$ with \[\left(w(1)<\cdots<w(i_1)| w(i_1+1)<\cdots<w(i_2)| \cdots|w(i_{k-1}+1)<\cdots<w(i_k)\right).\] Furthermore, there are times in the paper where we need to consider coset representatives not in $W^P$; in those instances we write \[\left(w(1),\cdots,w(i_1)| w(i_1+1),\cdots,w(i_2)| \cdots|w(i_{k-1}+1),\cdots,w(i_k)\right)\] where the entries between the vertical bars may be interchanged.

\subsection{Chains}
An edge in the moment graph corresponds to a torus stable curve of a fixed degree. A chain along those edges corresponds to a torus stable curve where the degree is the sum of the edge degrees in the chain. So, studying chains in the moment graph gives us information about curves in the flag variety. Here we will follow the exposition of \cite{FW} and specialize to the case of partial flags. We say that two unequal elements $v$ and $w$ in $W^{P}$ are {\bf adjacent} if there is a reflection $s_{e_{{l}}-e_{m}} \in W$ such that $w=vs_{e_l-e_m}$. The reflection $e_{l}-e_{m}$ is the sum of simple reflections in $R^+\backslash R^+_P$. That is, if $i_{a-1}+1 \leq l \leq i_a$ and $i_{b-1}+1 \leq m\leq i_{b}$ then define
\begin{eqnarray*}
d(v,w)&=&e_l-e_m+\Delta_P=(e_{i_a}-e_{i_a+1})+\cdots+(e_{i_{b-1}}-e_{i_{b-1}+1})+\Delta_P\\
&=&(\overbracket{0,\cdots,0}^{a-1},\overbracket{1,\cdots,1}^{b-a},\overbracket{0,\cdots,0}^{k+1-b}).
\end{eqnarray*}

Define a {\bf chain} $\mathcal{C}$ from $v$ to $w$ in $W^{P}$ to be a sequence $u_0,u_1, \cdots, u_r$ in $W^{P}$ such that $u_{i-1}$ and $u_i$ are adjacent for $1 \leq i \leq r$ and $u_0 \leq v$ and $w \leq u_r$. We say that the chain {\bf originates} at $u_0$ and {\bf terminates} at $u_r$. For any chain $u_0,u_1, \cdots, u_r$ we define the {\bf degree} of the chain $\mathcal{C}$, denoted $\deg_{\mathcal{C}}(v,w)$, to be the sum of the degrees $d(u_{i-1},u_i)$ for $1 \leq i \leq r$. Note that there is a chain of degree 0 between $v$ and $w$ exactly when $w \leq v$.

\subsection{Quantum Cohomology}
Let $\QH^*(\Fl)$ denote the quantum cohomology ring of $\Fl$. The Schubert classes $\sigma_w$, $w \in W^{P}$, form a basis. Let $\sigma^w:=\sigma_w^\vee$ be the Poincare dual of $\sigma_w$ for any $w\in W^{P}$. Take a variable $q_j$ for each $i_j \in I$ with $1 \leq j \leq k$, and let $\mathbb{Z}[q]$ be the polynomial ring with these $q_j$ as indeterminates where $\deg q_j=i_{j+1}-i_{j-1}$. For a degree $d=(d_1, \cdots, d_k)$ that corresponds to $\sum_{j=1}^{k}d_j \sigma^{s_{i_j}} \in H_2(\Fl)$ (this is an integral sum of curve classes), we write $q^d=\Pi_{j=1}^{k}q_j^{d_j}$. The small quantum cohomology ring $\QH^*(\Fl)$ is a graded $\mathbb{Z}[q]$-module. The multiplication is given by

\[ \sigma^v \star \sigma_w= \sum_{u,d\geq0} c_{v^\vee,w}^{u,d}q^d\sigma_u\] where $c_{v^\vee,w}^{u,d}$ is the Gromov-Witten invariant that enumerates the degree $d$ rational curves. See \cite{buch:partial} for details.

\begin{remark}
The $q$ indeterminates would be $q_1, \cdots, q_k$ with the degree coming from $I$. So, different subsets of $I$ with $k$ elements correspond to the same set of indeterminates but different grading.
\end{remark}

\subsection{Hecke Product} The purpose of using the Hecke product comes from the work of Buch and Mihalcea in \cite{buch.m:nbhds}. In particular, they use the Heck product to calculate curve neighborhoods of Schubert varieties which are the closures of degree $d$ rational curves that intersect a given Schubert variety. The curve neighborhood behavior is intimately related to minimum quantum degrees, which motivates the use of Heck products in this manuscript.

The Weyl group $\Wlarge$ admits a partial order $\leq$ given by the {\bf Bruhat order}. Its covering relations are given by $w < ws_\alpha$ where $\alpha \in R^+$ is a root and $\ell(w) < \ell(w s_\alpha)$. We will use the {\bf Hecke product} on the Weyl group $\Wlarge$. For a simple reflection $s_i$ the product is defined by 
\begin{align*}
    w \cdot s_i =   \begin{cases}
                        w s_i & \text{ if $\ell(ws_i)>\ell(w)$;} \\
                        w & \text{otherwise.}
                    \end{cases}
\end{align*}
If $v=s_{i_1}s_{i_2}\cdot \cdots \cdot s_{i_t}$ then $w\cdot v= w \cdot s_{i_1} \cdot s_{i_2} \cdot  \ldots \cdot s_{i_t}$. It is shown in \cite{buch.m:nbhds}*{Section 3} that this product is independent of the reduced expression chosen for $v$. The Hecke product gives $\Wlarge$ a structure of an associative monoid; see, e.g.,~\cite{buch.m:nbhds}*{\S 3} for more details. 
For any parabolic group $P$, the Hecke product determines a left action $W \times W/\Wlarge_P \to W/\Wlarge_{P}$ defined by 
\[
    u \cdot (w \Wlarge_{P}) = (u \cdot w) \Wlarge_P. 
\]
See the paragraph following \cite{buch.m:nbhds}*{Proposition 3.2}.

\subsection{Fulton and Woodward's formula for minimal quantum degrees} Given any element $\tau \in \QH^*(\Fl)$, we say that $q^d$ {\bf occurs} in $\tau$ if the coefficient of $q^d\sigma_w$ is not zero for some $w$. The following result provides an equivalent definition to degrees in terms of chains in the Bruhat graph.

\begin{prop} \label{prop:chains} \cite{FW}*{Theorem 9.1}
Let $v,w \in W^{P}$, and let $d$ be a degree. The following are equivalent:
\begin{enumerate}
    \item There is a degree $c \leq d$ such that $q^c$ occurs in $\sigma^v \star \sigma_w$.
    \item There is a chain of degree $c \leq d$ between $v$ and $w$.
\end{enumerate}
\end{prop}

\section{Maya Diagrams}

In this section, we give the definition of Maya diagrams. Maya diagrams appear in different contexts. (see, e.g., \cites{DATE1989149,RL,clarkson2019cyclic}). We will also describe the Bruhat order in terms of Maya diagrams.

\begin{defn} \label{defn:MayaDiagrams}
Let $w \in W^{P}$. The Maya diagram $M^w$ corresponding to $w$ is an $(k+1)\times n$ grid with the southwest corner chosen to be $(1,1)$ box and we index with $(\mbox{rows},\mbox{columns})$. We place an $x$ in the $(y,w(i))$ position for $1 \leq j \leq k+1$, $1 \leq i \leq i_j,$ and $j \leq y \leq k+1$. We color the bottom $x$ of each column black with all other $x$'s blue. We denote the row indexed by $y$ as $m^w_y$. 
\end{defn}

Each row corresponds to an increasing interval in the permutation, and you can read out the one-line notation by following the black $x$'s row by row from the bottom.

\begin{example} \label{mayaex}The minimal length representatives $w=(1|5<9|10<11|4<6|2<7)$ and $v=(2|7<11|10<12|8<9|1<5)$ corresponds to the Maya diagrams
\[ M^w=\scalebox{.75}{\ytableausetup{centertableaux}
\begin{ytableau}
\bx   & \bx & x  & \bx& \bx   &\bx& \bx & x & \bx & \bx &\bx & x \\
\bx   & x &   & \bx& \bx   &\bx&x &  & \bx & \bx &\bx &  \\
\bx   &   &   & x  & \bx   & x &  &  & \bx &\bx  &\bx &\\
\bx   &   &   &    & \bx   &   &  &  & \bx &x    &x   &\\
\bx   &   &   &    & x     &   &  &  & x   &     &    &\\
  x   &   &   &    &       &   &  &  &     &     &    &
\end{ytableau}\ytableausetup{centertableaux}}  \mbox{ and }  
  M^v=\scalebox{.75}{\ytableausetup{centertableaux}
\begin{ytableau}
\bx  & \bx & x  & x &\bx  & x  & \bx &\bx  &\bx  &\bx & \bx   &\bx \\
 x  & \bx &   &  &x  &   & \bx &\bx  &\bx  &\bx & \bx   &\bx \\
   & \bx  &   &  &   &   & \bx &x    &x    &\bx & \bx   &\bx \\
   & \bx  &   &  &   &   & \bx &     &     &x   & \bx   &x   \\
   & \bx  &   &  &   &   & x   &     &     &    & x     &    \\
   & x    &   &  &   &   &     &     &     &    &       &
\end{ytableau}\ytableausetup{centertableaux}}.\] In $M^w$, a black $x$ is placed in the bottom row and the 1st column; two black $x$'s are placed in the 2nd row up in the 5th and 9th columns. This corresponds to the $1,5$ and $9$ in $w=(1|5<9|10<11|4<6|2<7)$.
\end{example}

Let $M^w$ be the Maya diagram corresponding to $w \in W^{P}$ and let $1 \leq y \leq k$. Let $\pi_y:W^{P} \rightarrow W^{P_{i_y}}$ denote the natural projection. Then $M^{\pi_y(w)}$ is a Maya diagram with two rows and $n$ columns, with the top row having an $x$ in each position and the bottom row is $m^w_y$.

\begin{example}
    {\[ M^w=\scalebox{.75}{\ytableausetup{centertableaux}
\begin{ytableau}
\bx   & \bx & x  & \bx& \bx   &\bx& \bx & x & \bx & \bx &\bx & x \\
\bx   & x &   & \bx& \bx   &\bx&x &  & \bx & \bx &\bx &  \\
\bx   &   &   & x  & \bx   & x &  &  & \bx &\bx  &\bx &\\
\bx   &   &   &    & \bx   &   &  &  & \bx &x    &x   &\\
\bx   &   &   &    & x     &   &  &  & x   &     &    &\\
  x   &   &   &    &       &   &  &  &     &     &    &
\end{ytableau}\ytableausetup{centertableaux}}  \mbox{ and }  
  M^{\pi_3(w)}=\scalebox{.75}{\ytableausetup{centertableaux}
\begin{ytableau}
\bx  & x & x  & x &\bx  & x  & x &x  &\bx  &\bx & \bx   &x \\
 x  &  &   &  &x  &   &  &  &x  &x &x    &
\end{ytableau}\ytableausetup{centertableaux}}.\]}
\end{example}

\subsection{Bruhat order with Maya diagrams}\label{subs:BruhatOrder}
We begin the subsection with technical definitions.
\begin{defn}
Let $w,v \in W^P$. Let $M^w$ be the Maya diagram that corresponds to $w \in W^{P}$. 
\begin{enumerate}
\item Define 
\[ f(M^w,y,z):=\begin{cases} 
      x & y=j, z=w(i), \mbox{ for some } i,j \mbox{ with } 1 \leq i \leq i_j \mbox{ and }1\leq j \leq k+1; \\
     0 & \mbox{otherwise}
   \end{cases}.
\]
\item Define $S_y(M^w,z):=\#\{i: f(M^w,y,i)=x \mbox{ for }1 \leq i \leq z \}$. 
\item We say that $M^w \leq M^v$ if  $S_y(M^w,z) \geq S_y(M^v,z)$ for all $y$ and $z$ such that $1 \leq z \leq n$ and $1 \leq y \leq k+1$. 
\end{enumerate}
\end{defn}

\begin{example}
Consider $w=(1|5<9|10<11|4<6|2<7)$. Then we have

\[ M^w=\scalebox{.75}{\ytableausetup{centertableaux}
\begin{ytableau}
\bx   & \bx & x  & \bx& \bx   &\bx& \bx & x & \bx & \bx &\bx & x \\
\bx   & x &   & \bx& \bx   &\bx&x &  & \bx & \bx &\bx &  \\
\bx   &   &   & x  & \bx   & x &  &  & \bx &\bx  &\bx &\\
\bx   &   &   &    & {\color{brown}x}   &   &  &  & \bx &x    &x   &\\
\bx   &   &   &    & x     &   &  &  & x   &     &    &\\
  x   &   &   &    &       &   &  &  &     &     &    &
\end{ytableau}\ytableausetup{centertableaux}}. \]

We are considering the conditions where $y=j, z=w(i), \mbox{ for some } i,j \mbox{ with } 1 \leq i \leq i_j \mbox{ and }1\leq j \leq k+1$. Here, $i_1=1$, $i_2=3$ and $w(2)=5$. This means the fifth column is marked with an $x$ in the second row since $1 \leq 2 \leq i_2$ but $2>i_1$. Since $1 \leq 2 \leq i_j$ for $2 \leq j \leq k+1$, there is an $x$ in the fifth column and $j$th row for $2 \leq j \leq k+1.$ In particular, have that ${\color{brown}f(M^w,3,5)=x}$ which corresponds to the ${\color{brown}x}$ in $M^w$.

\end{example}

\begin{example} Recall the Maya diagrams from Example \ref{mayaex}.

\[ M^w=\scalebox{.75}{\ytableausetup{centertableaux}
\begin{ytableau}
\bx   & \bx & x  & \bx& \bx   &\bx& \bx & x & \bx & \bx &\bx & x \\
\bx   & x &   & \bx& \bx   &\bx&x &  & \bx & \bx &\bx &  \\
\bx   &   &   & x  & \bx   & x &  &  & \bx &\bx  &\bx &\\
{\color{brown} x}   &   &   &    &{\color{brown} x}   &   &  &  &{\color{brown} x} &x    &x   &\\
\bx   &   &   &    & x     &   &  &  & x   &     &    &\\
 {\color{violet} x}   &   &   &    &       &   &  &  &     &     &    &
\end{ytableau}\ytableausetup{centertableaux}}  \mbox{ and }  
  M^v=\scalebox{.75}{\ytableausetup{centertableaux}
\begin{ytableau}
\bx  & \bx & x  & x &\bx  & x  & \bx &\bx  &\bx  &\bx & \bx   &\bx \\
 x  & \bx &   &  &x  &   & \bx &\bx  &\bx  &\bx & \bx   &\bx \\
   & \bx  &   &  &   &   & \bx &x    &x    &\bx & \bx   &\bx \\
   & {\color{red}x}  &   &  &   &   &{\color{red}x} &     &     &x   & \bx   &x   \\
   & \bx  &   &  &   &   & x   &     &     &    & x     &    \\
 \colorbox{cyan}{\strut}  & x    &   &  &   &   &     &     &     &    &       &
\end{ytableau}\ytableausetup{centertableaux}}\]

Next we have a few examples of computations color coordinated to match with the $x$ that are being counted in the Maya diagrams. \[ {\color{violet} S_1(M^w,1)=1 } \geq 0 = S_1(M^v,1) \mbox{ and } {\color{brown} S_3(M^w,9)=3 } \geq {\color{red} 2 = S_3(M^v,9) }.\]

In this example, we have that $M^w \leq M^v$.
\end{example}

Next we present a proposition that relates the Bruhat order on $W^P$ with the partial order on Maya diagrams. This is another way of presenting the result in \cite{ProctBO}*{Theorem 5A}.

\begin{prop} \label{prop:bruhatorder}\cite{ProctBO}*{Theorem 5A}
Let $w,v \in W^{P}$. Then $w \leq v$ if and only if $M^w \leq M^v$.
\end{prop}

\section{Maya diagram combinatorics}

In the section, we will give the definition of the generalized rim hook rule and study chains in terms of Maya diagrams.

\subsection{Generalized rim hook rule} \label{subs:genrimhook}

Maya diagrams give a way to see a generalized rim hook rule. The generalized $(a,b)$-rim hook connects the combinatorics of the Maya diagram and Heck product to curves of degree \[(\overbracket{0,\cdots,0}^{a-1},\overbracket{1,\cdots,1}^{b-a},\overbracket{0,\cdots,0}^{k+1-b}).\] This can be thought of as a generalization of removing rim hooks on Young Tableau in the Grassmannian case that is presented in \cite{BERTRAM1999728}. We now define the generalized rim hook rule.

\begin{defn}
Let $v \in W^{P}$ and let $M^v$ be the corresponding Maya diagram. For $1 \leq y \leq k$ define \[ \phi(M^v,y):=\min \{z:f(M^v,y,z)=x \mbox{ and } f(M^v,y-1,z)=0 \},\] 
\[\psi(M^v,y):=\max \{z:f(M^v,y+1,z)=x \mbox{ and } f(M^v,y,z)=0 \}.\] 
\end{defn}

\begin{example}
Consider the following Maya diagram.
\[M^v=\scalebox{.75}{\ytableausetup{centertableaux}
\begin{ytableau}
\bx  & \bx & x  & x &\bx  & x  & \bx &\bx  &\bx  &\bx & \bx   &\bx \\
 x  & \bx &   &  &{\color{purple}x}  &   & \bx &\bx  &\bx  &\bx & \bx   &\bx \\
   & \bx  &   &  &   &   & \bx &x    &x    &\bx & \bx   &\bx \\
   & \bx  &   &  &   &   & \bx &     &     &{\color{green}x}   & \bx   &x   \\
   & \bx  &   &  &   &   & x   &     &     &    & x     &    \\
   & x    &   &  &   &   &     &     &     &    &       &
\end{ytableau}\ytableausetup{centertableaux}}.\]
Here ${\color{green}\phi(M^v,3)=10}$ and ${\color{purple} \psi(M^v,4)=5 }$ where the colors correspond to the $x$ in the definitions of $\phi$ and $\psi$.
\end{example}

Next we define the generalized rim hook rule.

\begin{defn} \label{defn:genrimhook}

Let $a,b$ be such that $1 \leq a<b \leq k+1$, and let $M^v$ be a Maya diagram for $v \in W^P$. We define the {\it generalized $(a,b)$-rim hook} as the Maya diagram obtained by the following process:

\begin{enumerate}
\item Let $M^v_{\uparrow a}:=M^v$. 

\item For $a \leq j \leq b-1$, define $M^v_{\uparrow j+1}$ from $M^v_{\uparrow j}$ by removing the $x$ in position $(j,\phi(M^v_{\uparrow j},j)).$

\item Let $M^v_{\downarrow b}:=M^v_{\uparrow b}$. 

\item For $b-1 \geq j \geq a$, define $M^v_{\downarrow j}$ from $M^v_{\downarrow j+1}$ by adding an $x$ to position $(j,\psi(M^v_{\downarrow j},j))$.

\item The completed generalized $(a,b)$-rim hook is given by $M^v_{\downarrow a}$.
\end{enumerate}
We refer to this process as the {\it $(a,b)$-rim hook rule.}
\end{defn}

Definition \ref{defn:genrimhook} is well-defined since we are ensuring that the number of $x$'s in each row is the same at the beginning and the end of the process and that any $x$, except those in the top row, has an $x$ above it.

\begin{remark}
    The definition \ref{defn:genrimhook} corresponds to the rim hook for Grassmannians (i.e., the case $k=1$). See \cite{FW}.
\end{remark}

\begin{example} \label{ex:grass} This example is in the Grassmannian $\Fl(\{0<8<12\};12)$ with \[v=(1<2<3<5<8<9<11<12) \in W^P.\] Here we will calculate the generalized $(1,2)$-rim hook. The symbols $\uparrow$ and $\downarrow$ describe the movement of the $x$'s.
\begin{eqnarray*}
M^v=\scalebox{.75}{\ytableausetup{centertableaux}
\begin{ytableau}
\bx&\bx&\bx & x   & \bx     &  x & x &\bx &\bx  &x & \bx&\bx  \\
x&x&x &    & x     &   &  &x & x & & x&x 
\end{ytableau}\ytableausetup{centertableaux}} &\overset{\strut}{\longrightarrow}&
\scalebox{.75}{\ytableausetup{centertableaux}
\begin{ytableau}
x&\bx& \bx       & x   & \bx & x  & x &\bx &\bx  &\downarrow & \bx         &\bx  \\
\uparrow &x& x       &    & x &   &  &x & x  & x & x         &x 
\end{ytableau}\ytableausetup{centertableaux}}
\end{eqnarray*}
\begin{eqnarray*}
&=&\scalebox{.75}{\ytableausetup{centertableaux}
\begin{ytableau}
x&\bx& \bx     & x  & \bx         & x  & x &\bx &\bx  &\bx & \bx &\bx \\
&x& x      &    & x         &   &  &x &x  &x & x &x 
\end{ytableau}.\ytableausetup{centertableaux}}
\end{eqnarray*}
\end{example}

\begin{example} The following is an example of a generalized $(2,6)$-rim hook in \[ \Fl(\{0<1<3<5<7<9<12\};12).\] Here $v=(2|3<8|10<12|9<11|1<5) \in W^P$.
\begin{eqnarray*}
M^v=\scalebox{.75}{\ytableausetup{centertableaux}
\begin{ytableau}
\bx&\bx&\bx & x   & \bx     &  x & x &\bx &\bx  &\bx & \bx&\bx  \\
x&\bx&\bx &    & x     &   &  &\bx &\bx  &\bx & \bx&\bx  \\
 &\bx&\bx&    &       &   &  &\bx & x   & \bx& x  &\bx \\
 &\bx&\bx&    &       &   &  &\bx&     & x  &    &x\\
 &\bx& x &    &       &   &  &x  &     &    &    & \\
 &x  &    &    &       &   &  &    &     &    &    &
\end{ytableau}\ytableausetup{centertableaux}} &\overset{\strut}{\longrightarrow}&
\scalebox{.75}{\ytableausetup{centertableaux}
\begin{ytableau}
x&\bx& \bx       & x   & \bx & x  & \downarrow &\bx &\bx  &\bx & \bx         &\bx  \\
\uparrow&\bx& x       &    & x &   & \downarrow &\bx &\bx  &\bx & \bx         &\bx  \\
 &\bx&\uparrow &    &            &   & x &\bx & x   & \bx& \downarrow &\bx \\
 &\bx&\uparrow &    &             &   &  &\bx&     & x  & x           & \downarrow \\
 &\bx&\uparrow&    &             &   &  &x  &     &    &             &x \\
 &x  &         &    &             &   &  &    &     &    &             &
\end{ytableau}\ytableausetup{centertableaux}}
\end{eqnarray*}
\begin{eqnarray*}
&=&\scalebox{.75}{\ytableausetup{centertableaux}
\begin{ytableau}
x&\bx& \bx       & x   & \bx & x  & \bx &\bx &\bx  &\bx & \bx         &\bx  \\
&\bx& x       &    & x &   & \bx &\bx &\bx  &\bx & \bx         &\bx  \\
 &\bx& &    &            &   & x &\bx & x   & \bx& \bx &\bx \\
 &\bx& &    &             &   &  &\bx&     & x  & x           & \bx \\
 &\bx& &    &             &   &  &x  &     &    &             &x \\
 &x  &         &    &             &   &  &    &     &    &             &
\end{ytableau}\ytableausetup{centertableaux}}.
\end{eqnarray*}
\end{example}

\subsection{Chains in terms of Maya diagrams } \label{ssect:Chains in terms of Maya diagrams}
The first lemma connects the generalized $(a,b)$-rim hook rule to the Hecke product and constructs a chain. We will now state a technical definition before describing Chains in terms of Maya diagrams.

\begin{defn}
    We say that two postive roots $e_a-e_b,e_c-e_d \in R^+$ intersect at most at an end point if $a<b\leq c<d$ or $c<d \leq a<b$.
\end{defn}

\begin{defn}
Let $\{ \lambda_q\}_{q=1}^{q}, \{\beta_j \}_{j=1}^J \subset R^+$. We say that $\sum \lambda_q \geq \sum \beta_j$ if $\sum \lambda_q - \sum \beta_j$ is a non-negative linear combination of positive roots.
\end{defn}

\begin{lemma} \label{lem:heckeprodfacts}
Let $M^v$ be a Maya diagram that corresponds to $v \in W^{P}$. Apply the $(a,b)$-rim hook rule to $M^v$ and call the resulting Maya diagram $M^{v'}$ where $v\in W^{P}$. Then we have the following: 

\begin{enumerate}
\item $s_{e_{i_{a-1}+1}-e_{i_b}}=s_{i_{a-1}+1}s_{i_{a-1}+2} \cdots s_{i_b-2}s_{i_b-1}s_{i_b-2} \cdots s_{i_{a-1}+2}s_{i_{a-1}+1}.$

\item $v'=v \cdot s_{e_{i_{a-1}+1}-e_{i_b}}.$

\item There exists a sequence of positive roots $\{ \beta_j \}_{j=1}^{J} \subset R^+$ such that 
\begin{enumerate}
\item $e_{i_{a-1}+1}-e_{i_b} \geq \sum \beta_j$;
\item any two elements of $\{ \beta_j \}_{j=1}^{J}$ overlap at most at an end point;
\item $v'=vs_{\beta_1}\ldots s_{\beta_J}$.
\end{enumerate}
\end{enumerate}
\end{lemma}

\begin{proof} We will prove each case individually.
\begin{enumerate} 

\item This is the result of a direct computation.

\item First, without loss of generality by re-indexing the word from Part (1), we only need to consider the case $s_{e_{1}-e_{i_n}}=s_1 s_2 \ldots s_{n-2}s_{n-1}s_{n-2}\ldots  s_2s_1$. If they exist, find the smallest indices $j_1, 1 \leq j_1 \leq  n-1$ and $j'_1, 1 \leq  j'_1 < n-1$ such that
\begin{enumerate}
\item $\ell(v \cdot s_1s_2\ldots s_{j_1})<\ell(v \cdot s_1s_2\ldots s_{j_1-1})$ and 
\item $\ell(v \cdot s_1s_2\ldots s_{n-2}s_{n-1}s_{n-2} \ldots s_{j'_1})<\ell(v \cdot s_1s_2\ldots s_{n-2}s_{n-1}s_{n-2} \ldots s_{j'_1-1})$.
\end{enumerate}
{\bf Case 1:} Suppose that $j_1$ exists but $j'_1$ does not. If $j_1=n-1$ then $v(n)=1$. But this implies $\ell(v \cdot s_1s_2\ldots s_{n-2}s_{n-1}s_{n-2})<\ell(v \cdot s_1s_2\ldots s_{n-2}s_{n-1})$ which implies $j'_1$ exists. So, $1 \leq j_1<n-1$. Then we have the following
\begin{eqnarray*}
v \cdot s_{e_{1}-e_{n}}&=&v \cdot s_1 s_2 \ldots s_{n-2}s_{n-1}s_{n-2}\ldots  s_2s_1 \\
&=& v \cdot s_1 s_2 \ldots s_{j_1-1} s_{j_1+1} \ldots s_{n-2}s_{n-1}s_{n-2}\ldots  s_2s_1\\
&=& \left(v \cdot \left( s_{j_1+1} \ldots s_{n-2}s_{n-1}s_{n-2}\ldots s_{j_1+1}\right)\right) \left(s_1 s_2 \ldots s_{j_1-1}s_{j_1}s_{j_1-1}  s_2s_1 \right)\\
&=& \left(v \cdot s_{e_{j_1+1}-e_n}\right) \left(s_{e_1-e_{j_1+1}} \right).
\end{eqnarray*}
{\bf Case 2:} Suppose that $j_1$ does not exist but $j'_1$ does. Then we have the following
\begin{eqnarray*}
v \cdot s_{e_{1}-e_{n}}&=&v \cdot s_1 s_2 \ldots s_{n-2}s_{n-1}s_{n-2}\ldots  s_2s_1 \\
&=& v \cdot s_1 s_2 \ldots s_{n-2}s_{n-1}s_{n-2}\ldots s_{j'_1+1}s_{j'_1-1} \ldots s_2s_1\\
&=& \left(v \left(s_1 s_2 \ldots s_{j'_1-1} s_{j'_1} s_{j'_1-1} \ldots s_2 s_1\right)\right) \cdot \left(s_{j_1+1} \ldots s_{n-2}s_{n-1}s_{n-2}\ldots s_{j_1+1}\right)\\
&=& \left( v\left(s_{e_1-e_{j'_1+1}} \right) \right) \cdot \left(s_{e_{j_1+1}-e_n}\right).
\end{eqnarray*}
{\bf Case 3:} Suppose that $j_1$ and $j'_1$ both exist, $j_1>j'_1$, and $j_1 \neq n-1$. Then we have the following

{\tiny
\begin{eqnarray*}
&&v \cdot s_{e_{1}-e_{n}}\\
&=&v \cdot s_1 s_2 \ldots s_{n-2}s_{n-1}s_{n-2}\ldots  s_2s_1 \\
&=& v \cdot s_1 s_2 \ldots s_{j_1-1}s_{j_1+1} \ldots s_{n-2}s_{n-1}s_{n-2}\ldots s_{j'_1+1}s_{j'_1-1} \ldots s_2s_1\\
&=&\left(v  s_1s_2 \ldots s_{j'_1-1}s_{j'_1}s_{j'_1-1} \ldots s_2s_1 \right) \cdot \left(s_{j_1'+1}s_{j_1'+2} \ldots s_{j_1-1}s_{j_1+1} \ldots s_{n-2}s_{n-1}s_{n-2} \ldots s_{j'_1+1} \right)\\
&=&\left(\left(v s_1s_2 \ldots s_{j'_1-1}s_{j'_1}s_{j'_1-1} \ldots s_2s_1\right) \cdot \left(s_{j_1+1} \ldots s_{n-2}s_{n-1}s_{n-2} \ldots s_{j_1+1}  \right) \right)\left(s_{j'_1+1}s_{j'_1+2} \ldots s_{j_1-1}s_{j_1}s_{j_1-1} \ldots s_{j'_1+2}s_{j'_1+1}\right)\\
&=&\left(\left(vs_{e_1-e_{j'_1}}\right) \cdot \left(s_{e_{j_1+1}-e_n}\right) \right) \cdot \left(s_{e_{j'_1+1}-e_{j_1+1}} \right).
\end{eqnarray*}
}
{\bf Case 4:} Suppose that $j_1$ and $j'_1$ both exist, $j_1>j'_1$, and $j_1=n-1$. Then we have the following
\begin{eqnarray*}
v \cdot s_{e_{1}-e_{n}}&=&\left(vs_{e_1-e_{j'_1}}\right)  \left(s_{e_{j'_1+1}-e_{n}} \right).
\end{eqnarray*}
{\bf Case 5:} Suppose that $j_1$ and $j'_1$ both exist, $j_1<j'_1$. Then we have the following
\begin{eqnarray*}
v \cdot s_{e_{1}-e_{n}}&=&\left(\left(v \cdot s_{e_{j_1+1}-e_{j'_1+1}} \right)\cdot \left(s_{e_{j'_1+1}-e_n}\right) \right)\left(s_{e_1-e_{j_1}}\right).
\end{eqnarray*}
{\bf Case 6:} Suppose that $j_1$ and $j'_1$ both exist and $j_1=j'_1$. Then we have the following
\begin{eqnarray*}
v \cdot s_{e_{1}-e_{n}}&=&vs_{e_{j_1+1}-e_{n}}.
\end{eqnarray*}
A key observation is that the (non-Hecke) permutation multiplication of $v$ by $s_{e_1-e_{j_1}}$ or $s_{e_1-e_{j'_1}}$ appears. Therefore, after iterating using the 6 cases above, the Hecke product $v \cdot s_{e_1-e_n}$ may be written as a (non-Hecke) product of $v$ times reflection of the form $s_{e_1-e_{j_1}}$ or $s_{e_1-e_{j'_1}}$ re-indexed as appropriate. When chained together, these reflections correspond to $M^v_{\uparrow j}$ and $M^v_{\downarrow j}$ from Definition \ref{defn:genrimhook}. Therefore, $v'=v \cdot s_{e_{i_{a-1}+1}-e_{i_b}}.$

\item This is an immediate consequence of applying iterations of the 6 cases in Part (2). The result follows.
\end{enumerate}
\end{proof}


\begin{example} \label{ex:chain}
    Consider the partial flag variety $\Fl(\{0<1<3<5<7<9<12\},12)$. Let $v=(2|3<8|10<12|9<11|1<5) \in W_P$. Here we will show the application of a $(2,6)-$rim hook to $v$. The outcome is $v \cdot s_{e_2-e_{12}}$ by Part (1) of Lemma \ref{lem:heckeprodfacts}. The following is an example of Part (3) of Lemma \ref{lem:heckeprodfacts} and the procedure to produce the sequence reflections. 
    \begin{enumerate}
    \item For the first step we have the following two facts.
    \begin{enumerate}
\item The number 12 is the biggest number in entries 2 through 12 (recall entries 9 through 12 are suppressed). 
\item The number 3 is the smallest number in an entry before 12 and in entries 2 through 12.
    \end{enumerate}
So we will use the reflection $s_{e_2-e_5}$ to interchange 3 and 12 to find
\[(2|{\color{green}3},8|10,{\color{green}12}|9,11|1,5)s_{e_2-e_5}=(2|{\color{red}12},{\color{red}8}|{\color{red}10},3|9,11|1,5).\] This is viewed visually in terms of Maya diagrams. Here we start by labeling the black $x$'s with their initial position reading bottom to top and left to right. In rows 2 through 6, notice that the $x$ (labeled $5$) farthest to the right is in the 12th column and the $x$ (labeled $2$) farthest to the left with a label less than 5 is in the third column. So, we interchange columns 3 and 12 using $s_{e_2-e_5}$.  The entries $x2,x3,x4$ are ineligible to be considered in future steps.
\[ \scalebox{.75}{\ytableausetup{centertableaux}
\begin{ytableau}
\bx&\bx&\bx              &x10&\bx&x11&x12&\bx&\bx&\bx&\bx&\bx\\
x8 &\bx&\bx              &   &x9 &   &   &\bx&\bx&\bx&\bx&\bx\\
   &\bx&\bx              &   &   &   &   &\bx&x6 &\bx&x7 &\bx\\
   &\bx&\bx              &   &   &   &   &\bx&   &x4 &   &{\color{green}x5}\\
   &\bx&{\color{green}x2}&   &   &   &   &x3 &   &   &   &\\
   &x1 &                 &   &   &   &   &   &   &   &   &
\end{ytableau}\ytableausetup{centertableaux}}  \overset{s_{e_2-e_5}}{\longrightarrow}
  \scalebox{.75}{\ytableausetup{centertableaux}
\begin{ytableau}
\bx&\bx&\bx&x10&\bx&x11&x12&\bx            &\bx&\bx            &\bx&\bx\\
x8 &\bx&\bx&   &x9 &   &   &\bx            &\bx&\bx            &\bx&\bx\\
   &\bx&\bx&   &   &   &   &\bx            &x6 &\bx            &x7 &\bx\\
   &\bx&x5 &   &   &   &   &\bx            &   &{\color{red}x4}&   &\bx\\
   &\bx&   &   &   &   &   &{\color{red}x3}&   &               &   &{\color{red}x2}\\
   &x1 &   &   &   &   &   &               &   &               &   &
\end{ytableau}\ytableausetup{centertableaux}}.\]

\item For the second step we have the following two facts.
\begin{enumerate}
\item The number 11 is the biggest number in entries 5 through 12. 
\item The number 3 is the smallest number in an entry before 12 and in entries 5 through 12.
    \end{enumerate}
So we will use the reflection $s_{e_5-e_7}$ to interchange 3 and 11 to find
\[(2|{\color{red}12},{\color{red}8}|{\color{red}10},{\color{green}3}|9,{\color{green} 11}|1,5)s_{e_5-e_7}=(2|{\color{red}12},{\color{red}8}|{\color{red}10},{\color{red}11}|{\color{red} 9},3|1,5).\] This is viewed visually in terms of Maya diagrams. In rows 2 through 6, notice that the $x$ (labeled $7$) farthest to the right is in the 11th column and the $x$ (labeled $5$) farthest to the left  with a label less than 7 is in the 3rd column. So, we interchange columns 3 and 11 using $s_{e_5-e_7}$. The entries $x2,x3,x4,x5,x6$ are ineligible to be considered in future steps.

\[ \scalebox{.75}{\ytableausetup{centertableaux}
\begin{ytableau}
\bx&\bx                  &\bx              &x10&\bx&x11&x12&\bx            &\bx            &\bx            &\bx              &\bx\\
x8 &\bx                  &\bx              &   &x9 &   &   &\bx            &\bx            &\bx            &\bx              &\bx\\
   &\bx                  &\bx              &   &   &   &   &\bx            &x6             &\bx            &{\color{green}x7}&\bx\\
   &\bx                  &{\color{green}x5}&   &   &   &   &\bx            &               &{\color{red}x4}&                 &\bx\\
   &\bx                  &                 &   &   &   &   &{\color{red}x3}&               &               &                 &{\color{red}x2}\\
   &x1                   &                 &   &   &   &   &               &               &               &                 &
\end{ytableau}\ytableausetup{centertableaux}}  \overset{s_{e_5-e_7}}{\longrightarrow}
  \scalebox{.75}{\ytableausetup{centertableaux}
\begin{ytableau}
\bx&\bx&\bx&x10&\bx&x11&x12&\bx            &\bx&\bx            &\bx&\bx\\
x8 &\bx&\bx&   &x9 &   &   &\bx            &\bx&\bx            &\bx&\bx\\
   &\bx&x7 &   &   &   &   &\bx            &{\color{red}x6} &\bx            &\bx  &\bx\\
   &\bx&   &   &   &   &   &\bx            &   &{\color{red}x4}&{\color{red}x5} &\bx\\
   &\bx&   &   &   &   &   &{\color{red}x3}&   &               &   &{\color{red}x2}\\
   &x1 &   &   &   &   &   &               &   &               &   &
\end{ytableau}\ytableausetup{centertableaux}}.\]

\item For the third step we have the following two facts.
\begin{enumerate}
\item The number 7 is the biggest number in entries 7 through 12. 
\item The number 1 is the smallest number in an entry before 7 and in entries 7 through 12.
    \end{enumerate}
So we will use the reflection $s_{e_8-e_{12}}$ to interchange 1 and 7 to find
\[(2|{\color{red}12},{\color{red}8}|{\color{red}10},{\color{red}11}|{\color{red} 9},3|{\color{green}1},5)s_{e_8-e_{12}}=(2|{\color{red}12},{\color{red}8}|{\color{red}10},{\color{red}11}|{\color{red} 9},3|7,{\color{red}5}).\] This is viewed visually in terms of Maya diagrams. In rows 2 through 6, notice that the $x$ (labeled $12$) farthest to the right is in the 7th column and the $x$ (labeled $8$) farthest to the left  with a label less than 12 is in the 1st column. So, we interchange columns 1 and 7 using $s_{e_8-e_{12}}$. The entries $x2,x3,x4,x5,x6,x9,x10,x11,x12$ are ineligible to be considered in future steps.
\[ \scalebox{.75}{\ytableausetup{centertableaux}
\begin{ytableau}
\bx&\bx&\bx&x10&\bx&x11&{\color{green}x12}&\bx            &\bx&\bx            &\bx&\bx\\
{\color{green}x8} &\bx&\bx&   &x9 &   &   &\bx            &\bx&\bx            &\bx&\bx\\
   &\bx&x7 &   &   &   &   &\bx            &{\color{red}x6} &\bx            &x  &\bx\\
   &\bx&   &   &   &   &   &\bx            &   &{\color{red}x4}&{\color{red}x5} &\bx\\
   &\bx&   &   &   &   &   &{\color{red}x3}&   &               &   &{\color{red}x2}\\
   &x1 &   &   &   &   &   &               &   &               &   &
\end{ytableau}\ytableausetup{centertableaux}}  \overset{s_{e_8-e_{12}}}{\longrightarrow}
  \scalebox{.75}{\ytableausetup{centertableaux}
\begin{ytableau}
{\color{red}x12}&\bx&\bx&{\color{red}x10}&\bx    &{\color{red}x11}&\bx            &\bx&\bx            &\bx&\bx&\bx\\
 &\bx&\bx&   &{\color{red}x9}     &   & x8              &\bx&\bx            &\bx&\bx            &\bx\\
   &\bx&x7&   &               &   &               &\bx&{\color{red}x6} &\bx            &x  &\bx\\
   &\bx&   &   &   &   &   &\bx            &   &{\color{red}x4}&{\color{red}x5} &\bx\\
   &\bx&   &   &   &   &   &{\color{red}x3}&   &               &   &{\color{red}x2}\\
   &x1 &   &   &   &   &   &               &   &               &   &
\end{ytableau}\ytableausetup{centertableaux}}.\]

\item For the fourth step we have the following two facts.
\begin{enumerate}
\item The number 7 is the biggest number in entries 7 through 8. 
\item The number 3 is the smallest number in an entry before 7 and in entries 7 through 8.
    \end{enumerate}
So we will use the reflection $s_{e_7-e_{8}}$ to interchange 3 and 7 to find
\[(2|{\color{red}12},{\color{red}8}|{\color{red}10},{\color{red}11}|{\color{red} 9},{\color{green}3}|{\color{green}7},{\color{red}5})s_{e_7-e_{8}}=(2|{\color{red}12},{\color{red}8}|{\color{red}10},{\color{red}11}|{\color{red} 9},{\color{red}7}|{\color{red}3},{\color{red}5}).\] This is viewed visually in terms of Maya diagrams. In rows 2 through 6, notice that the $x$ (labeled $8$) farthest to the right is in the 7th column and the $x$ (labeled $7$) farthest to the left  with a label less than 8 is in the 3rd column. So, we interchange columns 3 and 7 using $s_{e_7-e_{8}}$. There are no more eligible $x$ to consider in rows 2 through 6.
\[ \scalebox{.75}{\ytableausetup{centertableaux}
\begin{ytableau}
{\color{red}x12}&\bx&\bx&{\color{red}x10}&\bx    &{\color{red}x11}&\bx            &\bx&\bx            &\bx&\bx&\bx\\
 &\bx&\bx&   &{\color{red}x9}     &   & {\color{green}x8}              &\bx&\bx            &\bx&\bx            &\bx\\
   &\bx&{\color{green}x7}&   &               &   &               &\bx&{\color{red}x6} &\bx            &x  &\bx\\
   &\bx&   &   &   &   &   &\bx            &   &{\color{red}x4}&{\color{red}x5} &\bx\\
   &\bx&   &   &   &   &   &{\color{red}x3}&   &               &   &{\color{red}x2}\\
   &x1 &   &   &   &   &   &               &   &               &   &
\end{ytableau}\ytableausetup{centertableaux}}  \overset{s_{e_7-e_8}}{\longrightarrow}
  \scalebox{.75}{\ytableausetup{centertableaux}
\begin{ytableau}
{\color{red}x12}&\bx&\bx&{\color{red}x10}&\bx    &{\color{red}x11}&\bx            &\bx&\bx            &\bx&\bx&\bx\\
 &\bx&{\color{red}x8}&   &{\color{red}x9}     &   & \bx              &\bx&\bx            &\bx&\bx            &\bx\\
   &\bx&&   &               &   & {\color{red}x7}            &\bx&{\color{red}x6} &\bx            &x  &\bx\\
   &\bx&   &   &   &   &   &\bx            &   &{\color{red}x4}&{\color{red}x5} &\bx\\
   &\bx&   &   &   &   &   &{\color{red}x3}&   &               &   &{\color{red}x2}\\
   &x1 &   &   &   &   &   &               &   &               &   &
\end{ytableau}\ytableausetup{centertableaux}}.\]

\end{enumerate}

Finally observe that $v\cdot s_{e_2-e_{12}}=vs_{e_2-e_5}s_{e_5-e_7}s_{e_8-e_{12}}s_{e_7-e_8}$ and $e_2-e_{12}=(e_{2}-e_5)+(e_{5}-e_7)+(e_8-e_{12})+(e_7-e_8)$.
\end{example}

\begin{lemma} \label{lem:Mayasimroot}
Let $M^v$ be a Maya diagram that corresponds to $v \in W^{P}$. Apply the $(a,b)$-rim hook rule to $M^v$ and call the resulting Maya diagram $M^{v'}$ where $v\in W^{P}$. Then there is a chain $\mathcal{C}$ originating at $v$ to terminating at $v'$ such that \[\deg_{\mathcal{C}}(v,v') \leq (\overbracket{0,\cdots,0}^{a-1},\overbracket{1,\cdots,1}^{b-a},\overbracket{0,\cdots,0}^{k+1-b}).\]
\end{lemma}

\begin{proof}
This follows immediately from Part (3) of Lemma \ref{lem:heckeprodfacts}.
\end{proof}

The next two definitions give conditions for when to apply generalized $(a,b)$-rim hooks to produce a chain of minimum degree.
\begin{defn}
Let $v,w \in W^{P}$ with corresponding Maya diagrams $M^v$ and $M^w$, respectively. We say that position $(y,z)$ in $M^v$ is Bruhat order incompatible with position $(y,z)$ in $M^w$ if $S_y(M^v,z) > S_y(M^w,z)$. Similarly, we say that the row $m_y^v$ in $M^v$ is Bruhat order incompatible with the row $m_y^w$ in $M^w$ if $S_y(M^v,z) > S_y(M^w,z)$ for some $1 \leq z \leq n$. If $S_y(M^v,z) \leq S_y(M^w,z)$ for all $1 \leq z \leq n$, then we say $m_y^w \leq m_y^v$.
\end{defn}

It is important to apply the $(a,b)$-rim hook to as many rows as possible that are Bruhat incompatible while not applying it to rows that are Bruhat compatible. This is to reduce the number of $(a,b)$-rim hooks needed, so a minimum quantum degree is calculated.

\begin{defn}
Let $M^v$ and $M^w$ be Maya diagrams with $v,w \in W^{P}$, respectively, with the following properties:
\begin{enumerate} 
\item Row $m_j^v$ in $M^v$ is Bruhat order incompatible with row $m_j^w$ in $M^w$ for $a \leq j \leq b-1$;
\item If $a>1$, then $m_{a-1}^w \leq m_{a-1}^v$.
\item $m_b^w \leq m_b^v$.
\end{enumerate}
When the above conditions apply, let $ \displaystyle  \mathcal{C}^{(a,b)}_{(v,w)}$ denote a chain of degree less than or equal to $(\overbracket{0,\cdots,0}^{a-1},\overbracket{1,\cdots,1}^{b-a},\overbracket{0,\cdots,0}^{k+1-b})$ that originates at $v$ and terminating at $v'$ where $M^{v'}$ is the result of applying an $(a,b)$-rim hook rule to $M^v$.
\end{defn}

\begin{example}
Consider the following two Maya diagrams.

\begin{eqnarray*}
M^w=\scalebox{.75}{\ytableausetup{centertableaux}
\begin{ytableau}
\bx &  \bx &\bx  & x   &  \bx  &\bx&\bx & x &\bx  &\bx  &\bx &x& x \\
\bx &  x &x  &    &  \bx  &\bx&\bx &  &\bx  &\bx  &\bx &&  \\
\bx &    &   &    & \bx   &x  &x   &  &\bx  &\bx  &\bx &&\\
\bx &    &   &    & x   &   &    &  &\bx  &x    &x   &&\\
\bx &    &   &    &      &   &    &  &x    &x     &    &&\\
x   &    &   &    &       &   &    &  &     &     &    &&
\end{ytableau} \ytableausetup{centertableaux}} \mbox{ and } M^v=\scalebox{.75}{\ytableausetup{centertableaux}
\begin{ytableau}
\bx&\bx&\bx&x    & \bx     & x  & x &\bx &\bx  &\bx & \bx&x&\bx  \\
x&\bx&\bx&    & x     &   &  &\bx &\bx  &\bx & \bx&&\bx  \\
 &\bx&\bx!&    &       &   &  &\bx & x   & \bx& x  &&\bx \\
 &\bx&\bx!&    &       &   &  &\bx! &     & x  &    &&x\\
 &\bx&x!  &    &       &   &  &x!   &     &    &    && \\
 &x  &   &    &       &   &  &  &     &    &    &&
\end{ytableau}\ytableausetup{centertableaux}}.
\end{eqnarray*} 
Here row $m_j^v$ in $M^v$ is Bruhat order incompatible with row $m_j^w$ in $M^w$ for $2 \leq j \leq 4.$ Also, $m_1^w \leq m_1^v$ and $m_5^w \leq m_5^v$. We use $!$ to mark the Bruhat order incompatible positions in $M^v$.
\end{example}

\begin{defn}

Let $M^v$ and $M^w$ be a Maya diagram corresponding to $v,w \in W^{P}.$. Let $v_0:=v$. Let $\mathcal{C}$ be a chain in $W^{P}$ in terms of Maya diagrams given by
\begin{eqnarray*}
\mathcal{C}:M^{v_0} \overset{\mathcal{C}^{(a_0,b_0)}_{(v_0,w)}}{\longrightarrow} M^{v_1} \overset{\mathcal{C}^{(a_1,b_1)}_{(v_1,w)}}{\longrightarrow} M^{v_2} \overset{\mathcal{C}^{(a_2,b_2)}_{(v_{2},w)}}{\longrightarrow}\cdots \overset{\mathcal{C}^{(a_{r-1},b_{r-1})}_{(v_{r-1},w)}}{\longrightarrow} M^{v_r}.
\end{eqnarray*} 
Define $\displaystyle  \mbox{Comp}_y\mathcal{C}$ to be the $y$th component of $\displaystyle \sum_{j=0}^{r-1} \deg \mathcal{C}^{(a_j,b_j)}_{(v_j,w)}$.

\end{defn}






The next definition is necessary to state a Lemma \ref{lem:lbmdeg} which states a lower bound for the minimal degrees of chains connecting $v$ to $w$. 

\begin{defn} \label{def:grassmindeg}
Let $M^v$ and $M^w$ be a Maya diagram corresponding to $v,w \in W^{P}.$. Let $1 \leq y \leq k$ and let $\pi_y:W^{P} \rightarrow W^{P_{i_y}}$ be the natural projection where $P_{i_y}$ is the maximal parabolic subgroup associated with $i_y$. Let $v_0:=v$. {\it Define $\deg_y(v,w)$} to be the {\it smallest} integer such that there is a chain in terms of Maya diagrams given by \begin{eqnarray*}
\mathcal{C}(y):M^{\pi_y(v_0)} \overset{\mathcal{C}^{(1,2)}_{(\pi_y(v_0),\pi_y(w))}}{\longrightarrow} M^{\pi_y(v_1)} \overset{\mathcal{C}^{(1,2)}_{(\pi_y(v_1),\pi_y(w))}}{\longrightarrow} \cdots \overset{\mathcal{C}^{(1,2)}_{(\pi_y(v_{{\deg_y(v,w)}-1}),\pi_y(w))}}{\longrightarrow} M^{\pi_y(v_{\deg_y(v,w)})} 
\end{eqnarray*}
has the property $M^{\pi_y(v_j)} \ngeq M^{\pi_y(w)}$ for $1 \leq j \leq \deg_y(v,w)-1$ and $M^{\pi_y(v_{\deg_y(v,w)})} \geq M^{\pi_y(w)}$.  
\end{defn}

In Definition \ref{def:grassmindeg}, this is a chain of two-row May diagrams where the bottom row corresponds to the $y$-th row of the original Maya diagram. This corresponds to a minimum quantum degree calculation in the Grassmannian case $\Fl(\{0<i_y<n \};n).$.

\begin{example}
See Example \ref{ex:grass} for an example of a chain in Definition \ref{def:grassmindeg}.
\end{example}

\begin{lemma} \label{lem:lbmdeg}
Let $M^v$ and $M^w$ be Maya diagrams corresponding to $v,w \in W^{P}.$. Let $\mathcal{C}$ be any chain from $v$ to $w$. Then
\[ (\deg_1(v,w), \cdots, \deg_k(v,w)) \leq \deg_{\mathcal{C}}(v,w). \]
\end{lemma}

\begin{proof}
If not, then for some $y$, $\deg_y(v,w)$ is not the smallest integer such that there is a chain in terms of Maya diagrams given by \begin{eqnarray*}
\mathcal{C}(y):M^{\pi_y(v_0)} \overset{\mathcal{C}^{(1,2)}_{(\pi_y(v_0),\pi_y(w))}}{\longrightarrow} M^{\pi_y(v_1)} \overset{\mathcal{C}^{(1,2)}_{(\pi_y(v_1),\pi_y(w))}}{\longrightarrow} \cdots \overset{\mathcal{C}^{(1,2)}_{(\pi_y(v_{{\deg_y(v,w)}-1}),\pi_y(w))}}{\longrightarrow} M^{\pi_y(v_{\deg_y(v,w)})} 
\end{eqnarray*}
has the property $M^{\pi_y(v_j)} \ngeq M^{\pi_y(w)}$ for $1 \leq j \leq \deg_y(v,w)-1$ and $M^{\pi_y(v_{\deg_y(v,w)})} \geq M^{\pi_y(w)}$.
\end{proof}

\section{Maya diagrams and the Bruhat order}
 It is not clear that $\deg_y(v,w)$ and the $y$th component of $\sum_{j=0}^{r-1} \deg \mathcal{C}^{(a_j,b_j)}_{(v_j,w)}$ are equal. This is because generalized $(a,b)$-rim hooks do not necessarily remove the first $x$ in a row and place an $x$ in the last open position in a particular row. We use Lemmas \ref{lem:noprob}, \ref{lem:noprob1}, and \ref{lem:noprob2} to address this question. The next lemma addresses the Bruhat compatibility of adding an $x$ to a row when applying a generalized $(a,b)$-rim hook rule.

\begin{lemma} \label{lem:noprob}
Let $M^v$ and $M^w$ be a Maya diagram corresponding to $v,w \in W^{P}.$. Consider
\[M^{v} \overset{\mathcal{C}^{(a,b)}_{(v,w)}}{\longrightarrow} M^{v'}. \]
Consider the $y$th rows where $a \leq y \leq b-1$. Suppose that the new $x$ in $m_y^{v'}$ that is not in $m_y^{v}$ is in position $(y,z_0)$. It follows that $S_{y}(M^{v'},z) \leq S_{y}(M^{w},z)$ for all $z \geq z_0$.
\end{lemma}

\begin{proof}
For a contradiction, assume that $S_y(M^{v'},z)>S_y(M^{w},z)$ for some $z \geq z_0$. Then, by the generalized $(a,b)$-rim hook rule, we must have \[S_{y+1}(M^{v'},z)=S_y(M^{v'},z)+(i_{y+1}-i_y).\] Also, \[ S_y(M^{w},z)+(i_{y+1}-i_y) \geq S_{y+1}(M^{w},z).\] It follows that $S_{y+1}(M^{v'},z)>S_{y+1}(M^{w},z).$. Therefore, $S_{b}(M^{v'},z)>S_{b}(M^{w},z)$ by repeating the previous argument. This is a contradiction since $m_b^{v'} \geq m_b^{w}$. It follows that $S_{y}(M^{v'},z) \leq S_{y}(M^{w},z)$ for $z \geq z_0$.
\end{proof}

The next lemma addresses the Bruhat compatibility of not necessarily removing the first $x$ in a row when applying a generalized $(a,b)$-rim hook rule.

\begin{lemma} \label{lem:noprob1}
Let $M^v$ and $M^w$ be a Maya diagram corresponding to $v,w \in W^{P}.$. Consider
\[M^{v} \overset{\mathcal{C}^{(a,b)}_{(v,w)}}{\longrightarrow} M^{v'}. \] 
Consider the $y$th rows where $a \leq y \leq b-1$. Suppose that the $x$ in $m_y^{v}$ that is not in $m_y^{v'}$ is in position $(y,z_0)$. It follows that $S_y(M^{v'},z)\leq S_y(M^{w},z)$ for all $z < z_0$.
\end{lemma}

\begin{proof}
 If $a=1$  then $f(M^v,y,z)=0$ for all $z<z_0$. So, $S_1(M^{v'},z)\leq S_1(M^{w},z)$ for all $z < z_0$.

Suppose $a>1$ and $z<z_0$. By the definition of $\mathcal{C}^{(a,b)}_{(v,w)}$, we know that $S_{a-1}(M^{v},z)\leq S_{a-1}(M^{w},z)$. It must be true that $S_y(M^{v'},z) = S_{a-1}(M^v,z)$ since no $x$ is added or removed from the $y$-th row that is before the position $(y,z_0)$. Also, observe that $S_y(M^w,z)$ increases as $y$ increases. Therefore, \[S_y(M^{v'},z) = S_{a-1}(M^v,z) \leq S_{a-1}(M^{w},z) \leq S_y(M^{w},z). \] The result follows.
\end{proof}

\begin{example} 
Consider the following two Maya diagrams.
\[ M^w=\scalebox{.75}{\ytableausetup{centertableaux}
\begin{ytableau}
\bx &  \bx &\bx  & x   &  \bx  &\bx&\bx & x &\bx  &\bx  &\bx &x& x \\
\bx &  x &x  &    &  \bx  &\bx&\bx &  &\bx  &\bx  &\bx &&  \\
\bx &    &   &    & \bx   &x  &x   &  &\bx  &\bx  &\bx &&\\
\bx &    &   &    & x   &   &    &  &\bx  &x    &x   &&\\
\bx &    &   &    &      &   &    &  &x    &x     &    &&\\
x   &    &   &    &       &   &    &  &     &     &    &&
\end{ytableau}\ytableausetup{centertableaux}}  \mbox{ and } 
  M^v=\scalebox{.75}{\ytableausetup{centertableaux}
\begin{ytableau}
\bx&\bx&\bx&x    & \bx     & x  & x &\bx &\bx  &\bx & \bx&x&\bx  \\
x&\bx&\bx&    & x     &   &  &\bx &\bx  &\bx & \bx&&\bx  \\
 &\bx&\bx&    &       &   &  &\bx & x   & \bx& x  &&\bx \\
 &\bx&\bx&    &       &   &  &\bx &     & x  &    &&x\\
 &\bx&x  &    &       &   &  &x   &     &    &    && \\
 &x  &   &    &       &   &  &  &     &    &    &&
\end{ytableau}\ytableausetup{centertableaux}}.\]

Then we have the following where $!$ indicates the Bruhat order incompatible positions and $\uparrow$ and $\downarrow$ describe the movement of the $x$'s for when a (2,5)-rim hook is applied. The entries in {\color{brown} brown} are Bruhat compatible with $M^w$ by Lemma \ref{lem:noprob} and the entries in {\color{purple} purple} are Bruhat compatible with $M^w$ by Lemma \ref{lem:noprob1}.
\begin{eqnarray*}
M^v=\scalebox{.75}{\ytableausetup{centertableaux}
\begin{ytableau}
\bx&\bx&\bx & x   & \bx     & x  & x &\bx &\bx  &\bx & \bx&x&\bx  \\
x&\bx&\bx &    & x     &   &  &\bx &\bx  &\bx & \bx&&\bx  \\
 &\bx&\bx!&    &       &   &  &\bx & x   & \bx& x  &&\bx \\
 &\bx&\bx!&    &       &   &  &\bx!&     & x  &    &&x\\
 &\bx& x! &    &       &   &  &x!  &     &    &    && \\
 &x  &    &    &       &   &  &    &     &    &    &&
\end{ytableau}\ytableausetup{centertableaux}} &\overset{(2,5)\mbox{-rim hook}}{\longrightarrow}&
\scalebox{.75}{\ytableausetup{centertableaux}
\begin{ytableau}
\bx&\bx& \bx       & x   & \bx & x  & x &\bx &\bx  &\bx & \bx         &x&\bx  \\
x&\bx& x       &    & \downarrow &   &  &\bx &\bx  &\bx & \bx         &&\bx  \\
 &{\color{brown} x}&\uparrow &    & {\color{purple} x}           &   &  & {\color{purple} x} & {\color{purple} x}   & {\color{purple} x}& {\color{purple} \downarrow} &&{\color{purple} x} \\
 &{\color{brown} x}&\uparrow &    &             &   &  &{\color{blue} x}&     & { x}  & {\color{purple} x}           && {\color{purple} \downarrow} \\
 &{\color{brown} x}&\uparrow&    &             &   &  & x  &     &    &             &&{\color{purple} x} \\
 &x  &         &    &             &   &  &    &     &    &             &&
\end{ytableau}\ytableausetup{centertableaux}}.\\
\end{eqnarray*}
\end{example}

We use Lemmas \ref{lem:noprob} and \ref{lem:noprob1} to produce the inequality stated in the next lemma. 

\begin{lemma}\label{lem:noprob2}
Let $M^v$ and $M^w$ be a Maya diagram corresponding to $v,w \in W^{P}$ and assume that $\deg_y(v,w)>0$. Let $v_0:=v$. Let $\mathcal{C}$ be a chain in $W^{P}$ in terms of Maya diagrams given by
\begin{eqnarray*}
\mathcal{C}:M^{v_0} \overset{\mathcal{C}^{(a_0,b_0)}_{(v_0,w)}}{\longrightarrow} M^{v_1} \overset{\mathcal{C}^{(a_1,b_1)}_{(v_1,w)}}{\longrightarrow} M^{v_2} \overset{\mathcal{C}^{(a_2,b_2)}_{(v_{2},w)}}{\longrightarrow}\cdots \overset{\mathcal{C}^{(a_{r-1},b_{r-1})}_{(v_{r-1},w)}}{\longrightarrow} M^{v_r} 
\end{eqnarray*}
where $M^{v_j} \ngeq M^w$ for $1 \leq j \leq r-1$ and $M^{v_r} \geq M^w$. It follows that $\mbox{Comp}_y\mathcal{C} \leq \deg_y(v,w).$
\end{lemma}

The strategy of the proof of Lemma \ref{lem:noprob2} is to show that if the number of $(a,b)$-rim hooks applied to the $y$th row of a Maya diagram is equal to $\deg_y(v,w)$ (as in Definition \ref{def:grassmindeg}), then the corresponding rows are Bruhat compatible. Furthermore, since the rows are Bruhat compatible, there is no need to apply another $(a,b)$-rim hook to that row, which gives $\deg_y(v,w)$ as an upper bound to the number of $(a,b)$-rim hooks needed to apply to the $y$th row.

\begin{proof}
Let $M^v$ and $M^w$ be a Maya diagram corresponding to $v,w \in W^{P}$ and assume that $\deg_y(v,w)>0$. Let $v_0:=v$. Let $1 \leq D \leq r$ be such that $\mathcal{C*}$ is a chain in $W^{P}$ in terms of Maya diagrams given by
\begin{eqnarray*}
\mathcal{C*}:M^{v_0} \overset{\mathcal{C}^{(a_0,b_0)}_{(v_0,w)}}{\longrightarrow} M^{v_1} \overset{\mathcal{C}^{(a_1,b_1)}_{(v_1,w)}}{\longrightarrow} M^{v_2} \overset{\mathcal{C}^{(a_2,b_2)}_{(v_{2},w)}}{\longrightarrow}\cdots \overset{\mathcal{C}^{(a_{D-1},b_{D-1})}_{(v_{D-1},w)}}{\longrightarrow} M^{v_D} 
\end{eqnarray*}
where $\mbox{Comp}_y\mathcal{C*} \leq \deg_y(v,w)$. If $m_y^{v_j} \geq m_y^w$ for some $1 \leq j \leq D-1$, then we are done since another $(a,b)$-rim hook will not be applied to the $y$th row. Suppose $m_y^{v_j} \ngeq m_y^w$ for $1 \leq j \leq D-1$ and $\mbox{Comp}_y\mathcal{C*} = \deg_y(v,w)$. Our aim is to show that $m_y^{v_D} \geq m_y^w$.

Suppose that the $x$ in $m_y^{v_{D-1}}$ that is not in $m_y^{v_D}$ is in position $(y,z_0)$ and that the new $x$ in $m_y^{v_D}$ that is not in $m_y^{v_{D-1}}$ is in position $(y,z_1)$. By Lemmas \ref{lem:noprob} and \ref{lem:noprob1}, if $z \leq z_0$ or $z_1 \leq z$ then $S_y(M^{v_D},z) \leq S_y(M^w,z)$. 

Let $z_0<z'<z_1.$. Let $\mathcal{C}(y)$ be a chain in terms of Maya diagrams given by \begin{eqnarray*}
\mathcal{C}(y):M^{\pi_y(v_0)} \overset{\mathcal{C}^{(1,2)}_{(\pi_y(v_0),\pi_y(w))}}{\longrightarrow} M^{\pi_y(v_1)} \overset{\mathcal{C}^{(1,2)}_{(\pi_y(v_1),\pi_y(w))}}{\longrightarrow} \cdots \overset{\mathcal{C}^{(1,2)}_{(\pi_y(v_{{\deg_y(v,w)}-1}),\pi_y(w))}}{\longrightarrow} M^{\pi_y(v_{\deg_y(v,w)})}. 
\end{eqnarray*}
where $M^{\pi_y(v_j)} \ngeq M^{\pi_y(w)}$ for $1 \leq j \leq \deg_j(v,w)-1$ and $M^{\pi_y(v_{\deg_y(v,w)})} \geq M^{\pi_y(w)}$. In the chain $\mathcal{C*}$ the number of $x$'s in row $y$ removed before $z_0$ is $\deg_y(v,w)$ and they are replaced behind $z_1$. Similarly, in the chain $\mathcal{C}(y)$ the number of $x$'s in row $1$ removed before $z_0$ is $\deg_y(v,w)$ and they are replaced behind $z_1$. This implies $S_y(M^{v_D},z')=S_1(M^{\pi_y(v_{\deg_y(v,w)})},z')$. Then \[S_y(M^{v_D},z')=S_1(M^{\pi_y(v_{\deg_y(v,w)})},z') \leq S_1(M^{\pi_y(w)},z')=S_y(M^w,z').\] Therefore, $m_y^{v_D} \geq m_y^w$.

The result follows since $y \notin \{a_i, a_{i}+1, \cdots, b_i-1\}$ for $D \leq i \leq r-1$ (that is, we do not apply another $(a,b)$-rim hook to the $y$th row).
\end{proof}

\section{Main Result}
We arrive at our main theorem by observing that $\deg_y(v,w)$ is bounded above and below by $\mbox{Comp}_y\mathcal{C}$ as stated in Theorem \ref{thm:lowerboundach}. We provide an example of calculating the minimum quantum degree using Maya diagrams in Example \ref{ex:finalex} and state a chain that yields a curve of minimum degree.
\begin{thm} \label{thm:lowerboundach}
Let $M^v$ and $M^w$ be a Maya diagram corresponding to $v,w \in W^{P}.$. Let $v_0:=v$. Let $\mathcal{C}$ be the chain in $W^{P}$ in terms of Maya diagrams given by
\begin{eqnarray*}
\mathcal{C}:M^{v_0} \overset{\mathcal{C}^{(a_0,b_0)}_{(v_0,w)}}{\longrightarrow} M^{v_1} \overset{\mathcal{C}^{(a_1,b_1)}_{(v_1,w)}}{\longrightarrow} M^{v_2} \overset{\mathcal{C}^{(a_2,b_2)}_{(v_{2},w)}}{\longrightarrow}\cdots \overset{\mathcal{C}^{(a_{r-1},b_{r-1})}_{(v_{r-1},w)}}{\longrightarrow} M^{v_r} 
\end{eqnarray*}
where $M^{v_j} \ngeq M^w$ for $1 \leq j \leq r-1$ and $M^{v_r} \geq M^w$. Then

\[ (\deg_1(v,w), \cdots, \deg_k(v,w))=\sum_{j=0}^{r-1} \deg \mathcal{C}^{(a_j,b_j)}_{(v_j,w)}. \]
\end{thm}

\begin{proof}
First note that $\mathcal{C}$ exists since repeated applications of generalized $(a,b)$-rim hooks result in a Maya diagram corresponding to the longest element in $W^P$. By Lemma \ref{lem:noprob2},  $ \displaystyle \mbox{Comp}_y\mathcal{C} \leq \deg_y(v,w)$. By Lemma \ref{lem:lbmdeg} we have that $ \displaystyle \deg_{y}(v,w) \leq \mbox{Comp}_y\mathcal{C}$.   Then $\displaystyle \mbox{Comp}_y\mathcal{C} \leq \deg_y(v,w) \leq \mbox{Comp}_y\mathcal{C}.$ The result follows. \end{proof}

\begin{example} \label{ex:finalex} Here we give an example of Theorem \ref{thm:lowerboundach}. Consider the following two Maya diagrams.
\[ M^w=\scalebox{.75}{\ytableausetup{centertableaux}
\begin{ytableau}
\bx &  \bx &\bx  & x   &  \bx  &\bx&\bx & x &\bx  &\bx  &\bx &x& x \\
\bx &  x &x  &    &  \bx  &\bx&\bx &  &\bx  &\bx  &\bx &&  \\
\bx &    &   &    & \bx   &x  &x   &  &\bx  &\bx  &\bx &&\\
\bx &    &   &    & x   &   &    &  &\bx  &x    &x   &&\\
\bx &    &   &    &      &   &    &  &x    &x     &    &&\\
x   &    &   &    &       &   &    &  &     &     &    &&
\end{ytableau}\ytableausetup{centertableaux}}  \mbox{ and } 
  M^v=\scalebox{.75}{\ytableausetup{centertableaux}
\begin{ytableau}
\bx&\bx&\bx&x    & \bx     & x  & x &\bx &\bx  &\bx & \bx&x&\bx  \\
x&\bx&\bx&    & x     &   &  &\bx &\bx  &\bx & \bx&&\bx  \\
 &\bx&\bx&    &       &   &  &\bx & x   & \bx& x  &&\bx \\
 &\bx&\bx&    &       &   &  &\bx &     & x  &    &&x\\
 &\bx&x  &    &       &   &  &x   &     &    &    && \\
 &x  &   &    &       &   &  &  &     &    &    &&
\end{ytableau}\ytableausetup{centertableaux}}.\]

Then we have the following where $!$ indicate the Bruhat order incompatible positions and $\uparrow$ and $\downarrow$ describe the movement of the $x$'s.
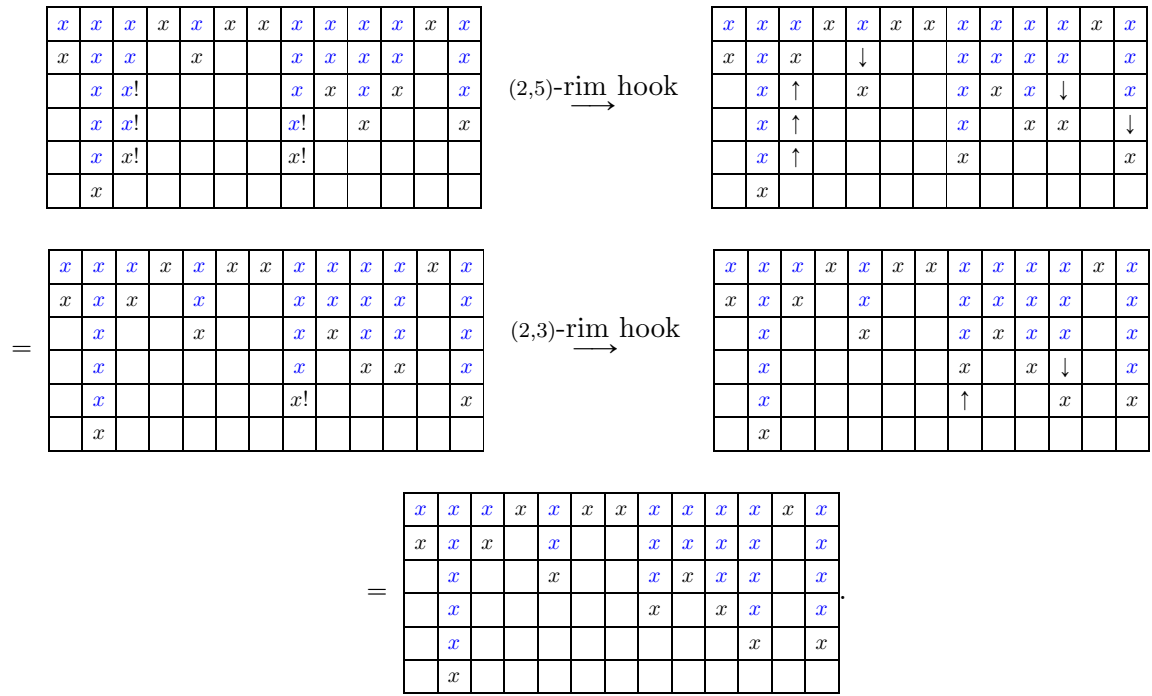
\begin{eqnarray*}
\hspace{10pt}\scalebox{.75}{\ytableausetup{centertableaux}
\begin{ytableau}
\bx&\bx&\bx & x   & \bx     & x  & x &\bx &\bx  &\bx & \bx&x&\bx  \\
x&\bx&\bx &    & x     &   &  &\bx &\bx  &\bx & \bx&&\bx  \\
 &\bx&\bx!&    &       &   &  &\bx & x   & \bx& x  &&\bx \\
 &\bx&\bx!&    &       &   &  &\bx!&     & x  &    &&x\\
 &\bx& x! &    &       &   &  &x!  &     &    &    && \\
 &x  &    &    &       &   &  &    &     &    &    &&
\end{ytableau}\ytableausetup{centertableaux}} &\overset{(2,5)\mbox{-rim hook}}{\longrightarrow}&
\scalebox{.75}{\ytableausetup{centertableaux}
\begin{ytableau}
\bx&\bx& \bx       & x   & \bx & x  & x &\bx &\bx  &\bx & \bx         &x&\bx  \\
x&\bx& x       &    & \downarrow &   &  &\bx &\bx  &\bx & \bx         &&\bx  \\
 &\bx&\uparrow &    & x           &   &  &\bx & x   & \bx& \downarrow &&\bx \\
 &\bx&\uparrow &    &             &   &  &\bx&     & x  & x           && \downarrow \\
 &\bx&\uparrow&    &             &   &  &x  &     &    &             &&x \\
 &x  &         &    &             &   &  &    &     &    &             &&
\end{ytableau}\ytableausetup{centertableaux}}\\
\end{eqnarray*}
\begin{eqnarray*}
\hspace{0pt}=\scalebox{.75}{\ytableausetup{centertableaux}
\begin{ytableau}
\bx&\bx& \bx       & x   & \bx         & x  & x &\bx &\bx  &\bx & \bx        &x&\bx \\
x&\bx& x       &    & \bx         &   &  &\bx &\bx  &\bx & \bx        &&\bx \\
 &\bx&         &    & x           &   &  &\bx & x   & \bx& \bx       & &\bx \\
 &\bx&         &    &             &   &  &\bx&     & x  & x          & &\bx \\
 &\bx&         &    &             &   &  &x!  &     &    &           &  &x   \\
 &x  &         &    &             &   &  &   &     &    &            & &
\end{ytableau}\ytableausetup{centertableaux}}
\hspace{0pt}&\overset{(2,3)\mbox{-rim hook}}{\longrightarrow}&\hspace{0pt}\scalebox{.75}{\ytableausetup{centertableaux}
\begin{ytableau}
\bx&\bx& \bx       & x   & \bx         & x  & x &\bx &\bx  &\bx & \bx        &x&\bx \\
x&\bx& x       &    & \bx         &   &  &\bx &\bx  &\bx & \bx        &&\bx \\
 &\bx&         &    & x           &   &  &\bx & x   & \bx& \bx        &&\bx \\
 &\bx&         &    &             &   &  &x&     & x  & \downarrow          & &\bx \\
 &\bx&         &    &             &   &  &\uparrow  &     &   &  x          & &x   \\
 &x  &         &    &             &   &  &   &     &    &             &&
\end{ytableau}\ytableausetup{centertableaux}}\\
\end{eqnarray*}
\begin{eqnarray*}
&&\hspace{0pt}=\scalebox{.75}{\ytableausetup{centertableaux}
\begin{ytableau}
\bx&\bx& \bx       & x   & \bx         & x  & x &\bx &\bx  &\bx & \bx        &x&\bx \\
x&\bx& x       &    & \bx         &   &  &\bx &\bx  &\bx & \bx        &&\bx \\
 &\bx&         &    & x           &   &  &\bx & x   & \bx& \bx       & &\bx \\
 &\bx&         &    &             &   &  &x&     & x  & \bx          & &\bx \\
 &\bx&         &    &             &   &  &  &     &   &  x           &&x   \\
 &x  &         &    &             &   &  &   &     &    &            & &
\end{ytableau}\ytableausetup{centertableaux}}.
\end{eqnarray*}
Thus the minimum quantum degree that appears in $\sigma^v \star \sigma_w$ is 
\[ (0,1,1,1,0)+(0,1,0,0,0)=(0,2,1,1,0).\]

Following the process in Example \ref{ex:chain}, the precise chain to describe this curve is given by \[ vs_{e_2-e_5}s_{e_5-e_7}s_{e_8-e_9}s_{e_7-e_8}s_{e_3-e_5}.\]
\end{example}

\section{Conflict of Interest}
On behalf of all authors, the corresponding author states that there is no conflict of interest.

\bibliographystyle{halpha}
\bibliography{bibliography}

\end{document}